\documentclass[11pt]{amsart}

\usepackage{amsmath,amssymb,amscd,verbatim}

\newcommand{\sub}{\subseteq}

\newcommand{\lra}{\Leftrightarrow}

\newcommand{\ra}{\Rightarrow}

\newcommand{\sm}{\setminus}

\newcommand{\La}{\Lambda}
\newcommand{\lam}{\lambda}

\newcommand{\tmax}{t\op{-Max}}

\newcommand{\Na}{\op{Na}}

\newcommand{\op}{\operatorname}

\newcommand{\nd}{\noindent}

\newtheorem{theorem}{Theorem}[section]

\newtheorem{lemma}[theorem]{Lemma}

\newtheorem{proposition}[theorem]{Proposition}

\newtheorem{corollary}[theorem]{Corollary}

\theoremstyle{definition}

\newtheorem{example}[theorem]{Example}

\newtheorem{question}[theorem]{Question}

\begin{document}

    \title{$w$-Divisoriality in Polynomial Rings}



\author{Stefania Gabelli}
\author{Evan Houston}
\author{Giampaolo Picozza}

\address[Stefania Gabelli and Giampaolo Picozza]{Dipartimento di Matematica, Universit\`{a} degli Studi Roma
Tre,
Largo S.  L.  Murialdo,
1, 00146 Roma, Italy}
\address[Evan Houston]{Department of Mathematics\\
 University of North Carolina at Charlotte\\
 Charlotte, NC 28223 U.S.A.}
\email[Stefania Gabelli]{gabelli@mat.uniroma3.it}

\email[Evan Houston] {eghousto@uncc.edu}

\email[Giampaolo Picozza]{picozza@mat.uniroma3.it}
\thanks{The second author was supported by a visiting grant from
GNSAGA of INdAM (Istituto Nazionale di Alta Matematica).}

\date{}




\begin{abstract}

We extend the Bass-Matlis characterization of local Noetherian
divisorial domains to the non-Noetherian case.  This result is
then used to study the following question: If a domain $D$ is
\emph{$w$-divisorial}, that is, if each $w$-ideal of $D$ is
divisorial, then is $D[X]$ automatically $w$-divisorial?  We show
that the answer is yes if $D$ is either integrally closed or Mori.
\end{abstract}

\maketitle

\section*{Introduction}

Throughout this paper, $D$ denotes an integral domain with
quotient field $K$.  For a nonzero fractional ideal $I$ of $D$, we
set $I^{-1} = (D:I) = \{x \in K \mid xI \sub D\}$, $I_v =
(I^{-1})^{-1}$, $I_t = \bigcup \{J_v \mid J \text{ is a nonzero
finitely generated subideal of I}\}$, and $I_w = \bigcup (I:J)$,
where the union is taken over all finitely generated ideals $J$ of
$D$ with $J_v=D$. An ideal $I$ is said to be a divisorial (resp.
$t$-, $w$-) ideal if $I=I_v$ (resp., $I=I_t$, $I=I_w$).  We assume
familiarity with properties of these star operations.  In
particular, we shall use the fact that each nonzero, nonunit
element of $D$ is contained in a maximal $t$-ideal, and we denote
the set of maximal $t$-ideals of $D$ by $\tmax(D)$.

A domain $D$ is said to be \emph{divisorial} if each of its
nonzero ideals is divisorial.  These domains have been studied by
Bass \cite{ba}, Matlis \cite{ma}, Heinzer \cite{h}, Bazzoni and
Salce \cite{bs}, and Bazzoni \cite{b}.  In \cite{GE} the first
author and S. El Baghdadi studied domains in which each $w$-ideal
of $D$ is divisorial and dubbed such domains \emph{$w$-divisorial}
domains. In part this work is a sequel to that paper.  Our main
goal is to study whether $w$-divisoriality transfers from $D$ to
the polynomial ring $D[X]$.  Although we do not give a definitive
answer, we do show that the property transfers in many cases, and
we analyze the difficulties in the general case.

In Section 1, we extend the Bass-Matlis characterization of
Noetherian divisorial domains to the non-Noetherian setting.
Recall that a nonzero ideal $I$ of $D$ is said to be
\emph{$m$-canonical} if $(I:(I:J))=J$ for each nonzero ideal $J$
of $D$.  We show in Theorem~\ref{t:divequiv} that if $D$ is a
local domain with nonprincipal maximal ideal $M$, then $D$ is
divisorial if and only if $(M:M)$ is a 2-generated $D$-module and
$M$ is an $m$-canonical ideal of $(M:M)$.

Recall \cite{AZ} that $D$ is \emph{weakly Matlis} if each nonzero
element of $D$ is contained in only finitely many maximal
$t$-ideals of $D$ and each nonzero prime $t$-ideal of $D$ is
contained in a unique maximal $t$-ideal. One of the main results
in \cite{GE} is that $D$ is $w$-divisorial if and only if $D$ is
weakly Matlis and $D_M$ is divisorial for each maximal $t$-ideal
$M$ of $D$.  We therefore devote Section 2 to a study of the
weakly Matlis property in polynomial rings. We prove that $D[X]$
is weakly Matlis if and only if $D$ is weakly Matlis and each
upper to zero in $D[X]$ is contained in a unique maximal $t$-ideal
of $D[X]$, and we give examples of weakly Matlis domains $D$ for
which $D[X]$ is not weakly Matlis. Finally, in Section 3 we
consider the main question of the paper. We are able to show that
if $D$ is $w$-divisorial and either integrally closed or Mori,
then $D[X]$ is also $w$-divisorial.


\section{Divisoriality}

We begin with a compilation of some results of Bazzoni-Salce
\cite[Proposition 5.4, Lemma 5.5, Theorem 5.7]{bs}.  Recall that a
domain $D$ is \emph{$h$-local} if each nonzero element of $D$ is
contained in only finitely many maximal ideals of $D$ and each
nonzero prime ideal of $D$ is contained in a unique maximal ideal.

\begin{theorem} \label{t:bs} [{Bazzoni-Salce}] Let $D$ be a domain. Then
$D$ is divisorial if and only if $D$ is $h$-local and $D_M$ is
divisorial for each maximal ideal $M$ of $D$. If $(D,M)$ is a
local divisorial domain with $M$ principal, then $D$ is a
valuation domain. Finally, if $(D,M)$ is a local divisorial domain
with $M$ nonprincipal, then $M^{-1}=(M:M)$ (so that $M^{-1}$ is an
overring of $D$), and exactly one of the following cases occurs:
\begin{enumerate}
\item $(M:M)$ is local with maximal ideal $N$ properly containing $M$,
in which case $N/M$ is simple both as a $D$-module and as a
$(M:M)$-module, $N^2 \sub M$, and $(D:N)=N$;
\item $(M:M)$ is a Pr\"ufer domain with exactly two maximal ideals
$N_1$ and $N_2$, and $M=N_1 \cap N_2$;
\item $(M:M)$ is a valuation domain with maximal ideal $M$.
\end{enumerate}
\end{theorem}

As noted in \cite{b}, (the first part of) this theorem effectively
reduces the question of divisoriality of a domain to the local
case.  It is well known that a local Noetherian domain $(R,M)$ is
divisorial if and only if $R$ has dimension one and $M^{-1}$ is a
2-generated $R$-module \cite[Theorems 6.2, 6.3]{ba} and
\cite[Theorem 3.8]{ma}. (See \cite{q} for an extension to the Mori
case.) Our next result gives a characterization of divisoriality
in the general case.  In the statement of the theorem, we assume
that the maximal ideal $M$ of the local domain $D$ is
nonprincipal. This assumption results in no loss of generality
since by Theorem~\ref{t:bs} if $M$ is principal, then $D$ is a
valuation domain (and this case is well understood).

\begin{theorem} \label{t:divequiv}
Let $D$ be a local domain with nonprincipal maximal ideal $M$, and
assume that $D$ is not a valuation domain. Then $D$ is a
divisorial domain if and only if $(M:M)$ is a 2-generated
$D$-module and $M$ is an $m$-canonical ideal of $(M:M)$.
\end{theorem}

\begin{proof} We use the notation of Theorem~\ref{t:bs}.  Suppose
that $D$ is divisorial.  Then $M$ is divisorial, and so $M^{-1}
\supsetneqq D$.  Since $M$ is not principal then $M^{-1}=MM^{-1}$,
so that $(M:M)=M^{-1}$. Set $T=(M:M)$. By \cite[Remark after Lemma
5.5]{bs}, $T$ is a 2-generated $D$-module (in fact, $T=D+Du$ for
each $u \in T \setminus D$).  Since $D$ is divisorial, the ideal
$D$ is $m$-canonical, and so $M=(D:T)$ is $m$-canonical by
\cite[Proposition 5.1]{hhp}.

Conversely, assume that $T$ is a 2-generated $D$-module and that
$M$ is an $m$-canonical ideal of $T$.  Since $D$ is assumed not to
be a valuation domain, $M$ is not $m$-canonical in $D$ by
\cite[Proposition 4.1]{bhlp}; hence $T$ properly contains $D$.
Thus $T/M$ is a 2-dimensional vector space over $D/M$, from which
it follows easily that $T=D+Dx$ for each $x \in T \setminus D$. In
particular, there are no rings properly between $D$ and $T$. Since
$T$ is integral over $D$, \cite[Corollary 2.2]{gh} implies that
$T$ has at most two maximal ideals.

Case 1. Suppose that $T$ is actually local. There are then two
subcases: $M$ is the maximal ideal of $T$ or $T$ is local with
maximal ideal properly containing $M$.

Case 1a.  Suppose that $M$ is the maximal ideal of $T$.  Then,
since $M$ is an $m$-canonical ideal of $T$, $T$ is a valuation
domain by \cite[Proposition 4.1]{bhlp}. Since there are no rings
properly between $D$ and $T$, \cite[Theorem 2.4]{gh} then implies
that $T$ is the unique minimal overring of $T$.  By
\cite[Proposition 2.7]{b}, $D$ is divisorial.

Case 1b. Now suppose that $T$ is local with maximal ideal $N
\supsetneqq M$. Since $M$ is an $m$-canonical ideal of $T$, by
\cite[Lemma 4.1]{hhp}, $M$ has a unique minimal fractional
overideal. We shall use this to show that $T$ is the unique
minimal overring of $D$. First, it follows from \cite[Lemma
2.1]{gh} that there are no ideals properly between $N$ and $M$ in
$T$.  Hence $N$ must be the unique minimal overideal of $M$.  Now,
let $x \in K \sm D$ (where $K$ is the common quotient field of $D$
and $T$). If $x \in T$, then we showed above that $T=D+Dx$. If $x
\notin T=(M:M)$, then $Mx \nsubseteq M$, and so the fractional
ideal $M+Mx$ properly contains $M$.  Hence $M+Mx \supseteq N$. Let
$v \in N \sm M$. Then $v \in M+Mx \sub D+Dx$, whence $T=D+Dv \sub
D+Dx$. Hence $T$ is indeed the unique minimal overring of $D$. To
show that $D$ is divisorial in this case, we consider a
nonprincipal ideal $I$ of $D$.  We claim that $I^{-1}=(D:I)$ is
also nonprincipal.  Suppose, on the contrary, that $I^{-1}$ is
principal.  Then $I_v$ is also principal, and we have $I_vI^{-1} =
D \nsubseteq M$.  Since $M$ is divisorial in $D$, it follows that
$II^{-1} \nsubseteq M$. However, this implies that $I$ is
invertible, hence principal, a contradiction. By \cite[Lemma
2.5]{b}, for any nonprincipal fractional ideal $J$ of $D$, $J$ is
an ideal of $T$. It is also easy to see that for such $J$ we have
$(D:J)=(M:J)$.  Thus, since $M$ is an $m$-canonical ideal of $T$,
we have $(D:(D:I))=(M:(D:I))=(M:(M:I))=I$.  Thus $D$ is
divisorial, as desired.

Case 2. Suppose that $T$ has two maximal ideals, say $N_1$ and
$N_2$.  We first claim that $M=N_1N_2$.  To verify this, recall
that $T/M$ is a 2-dimensional vector space over $D/M$, and
consider the chain of subspaces $T/M \supsetneqq N_1/M \supsetneqq
N_1N_2/M \supseteq (0)$.  We then have $MT_{N_i}=N_iT_{N_i}$ for
$i=1,2$. We also have by \cite[Proposition 5.5]{hhp} that
$MT_{N_i}$ is an $m$-canonical ideal of $T_{N_i}$ and therefore by
\cite[Proposition 4.1]{bhlp} that $T_{N_i}$ is a valuation domain.
Thus $T$ is a Pr\"ufer domain.  It now follows from \cite[Theorem
2.4]{gh} that $T$ is the unique minimal overring of $D$.  That $D$
is divisorial now follows from the same argument given in the
preceding paragraph.
\end{proof}

Let $F \sub K$ be fields with $[K:F]=2$, and let
$D=F+(X,Y)K[X,Y]_{(X,Y)}$.  Then $D$ is Noetherian and
2-dimensional, hence not divisorial, but the maximal ideal $M$ of
$D$ is such that $M^{-1}=(M:M)=K[X,Y]_{(X,Y)}$ is a 2-generated
$D$-module (but $M$ is not $m$-canonical in $(M:M)$). It would be
interesting to have a (necessarily non-Noetherian) 1-dimensional
example of a nondivisorial local domain $(D,M)$ with $(M:M)$ a
2-generated $D$-module.


\section{Weakly Matlis polynomial rings}


Let $D$ be a domain with quotient field $K$. For $h \in K[X]$, we
denote by $c(h)$ the $D$-ideal generated by the coefficients of
$h$.  We set $N=\{h \in D[X]
 \mid c(h)=D\}$.  The ring $D[X]_N$ is called the Nagata ring of $D$; this
 ring is also denoted by $D(X)$.
As in \cite{K}, we set $N_v= \{h\in D[X] \mid c(h)_v=D\}$ and call
the domain $D[X]_{N_v}$ the Nagata ring of $D$ with respect to the
$v$-operation.

We have $\op{Max}(D[X]_{N_v})=\{MD[X]_{N_v} \mid M\in \tmax(D)\}$
\cite[Proposition
2.1 (2)]{K} and, for $Q= MD[X]_{N_v}$, $(D[X]_{N_v})_{Q}= D[X]_{MD[X]}=
D_{M}(X)$.

Let $\Lambda$ be a nonempty set of prime ideals of a domain $D$.
Following \cite{AZ}, we say that $\Lambda$ has \emph{finite
character} if each nonzero element of $D$ belongs to at most
finitely many members of $\Lambda$ and that $\Lambda$ is
\emph{independent} if no two members of $\Lambda$ contain a common
nonzero prime ideal.  Thus $D$ is \emph{$h$-local} if
$\op{Max}(D)$ is independent of finite character \cite{ma}, and
$D$ is weakly Matlis if $\tmax(D)$ is independent of finite
character.  (Note that, since a prime ideal minimal over a nonzero
principal ideal is automatically a $t$-ideal, $\tmax(D)$ is
independent if and only if no two of its members contain a common
prime $t$-ideal.)

In the following lemma, the equivalence of (1) and (2) is proved
in \cite[Proposition 4.2]{km}.  (The statement of their result
includes the hypothesis that $D$ be integrally closed, but their
proof does not make use of this.)  We add a third equivalence.

\begin{lemma}\label{finchar} The following conditions are equivalent
for a domain $D$.
    \begin{itemize}
    \item[(1)] $D$ has $t$-finite character.
        \item[(2)] $D[X]$ has $t$-finite character.
        \item[(3)] $D[X]_{N_v}$ has finite character.
            \end{itemize}
\end{lemma}
\begin{proof} $(1)\lra(3)$.  This follow from that fact that
$\op{Max}(D[X]_{N_v}) = \{MD[X]_{N_v} \mid M \in \tmax{D}\}$
\cite[Proposition 2.1 (2)]{K}.
\end{proof}

The situation is not so simple for independence.

    \begin{proposition} \label{wm} The following conditions are
    equivalent for a domain $D$.
    \begin{itemize}
\item[(1)] $\tmax(D)$ is independent (respectively, $D$ is weakly
Matlis), and each upper to zero in $D[X]$ is contained in a unique
maximal $t$-ideal of $D[X]$.
        \item[(2)] $\tmax(D[X])$ is independent (respectively,
        $D[X]$ is weakly Matlis).
        \item[(3)] $\op{Max}(D[X]_{N_v})$ is independent (respectively,
        $D[X]_{N_v}$ is $h$-local).
        \end{itemize}
        \end{proposition}
\begin{proof} In view of Lemma~\ref{finchar}, we need prove only
the ``independent'' equivalences.

$(1)\ra(2)$.  Let $Q_1 \ne Q_2$ be maximal $t$-ideals of $D[X]$.
If either of the $Q_i$ is an upper to zero, then (since uppers to
zero have height one) $Q_1 \cap Q_2$ cannot contain a nonzero
prime.  Hence we may assume that both $Q_i$ are extended from
maximal $t$-ideals of $D$ \cite[Proposition 1.1]{hz}. Independence
of $\tmax(D)$ then implies that $Q_1 \cap Q_2$ cannot contain any
extended primes, and the other hypothesis implies that $Q_1 \cap
Q_2$ cannot contain an upper to zero.  Hence $\tmax(D[X])$ is
independent.

$(2)\ra(3)$.  This follows from the fact that the maximal ideals
   of $D[X]_{N_v}$ are extended from the (non-upper to zero) maximal
 $t$-ideals of $D[X]$.

$(3)\ra (1)$.  Independence of $\tmax(D)$ follows from the fact
that maximal $t$-ideals of $D$ extend to maximal ideals of
$D[X]_{N_v}$. If an upper to zero $Q$ is contained in $N_1 \cap
N_2$, where $N_1,N_2$ are maximal $t$-ideals of $D[X]$, then
$QD[X]_{N_v}$ is a prime of $D[X]_{N_v}$ which is contained in the
maximal ideals $N_iD[X]_{N_v}$, $i=1,2$, a contradiction.
\end{proof}

Recall from \cite{hz} that $D$ is a \emph{UMT-domain} if each
upper to zero in $D[X]$ is a maximal $t$-ideal.

\begin{corollary} \label{c:umt} Suppose that $D$ has a unique
maximal $t$-ideal (equivalently, $D$ is a local domain whose
maximal ideal is a $t$-ideal) or is a UMT-domain. Then the
following conditions are equivalent.
\begin{itemize}
\item[(1)] $D$ is weakly Matlis.
\item[(2)] $D[X]$ is weakly Matlis.
\item[(3)] $D[X]_{N_v}$ is $h$-local.
\end{itemize}
\end{corollary}
\begin{proof}  The result follows easily from Proposition~\ref{wm},
in view of the fact that the hypothesis on $D$ guarantees that
each upper to zero in $D[X]$ is contained in a unique maximal
$t$-ideal.
\end{proof}


The polynomial ring over a weakly Matlis domain need not be weakly
Matlis, as Examples~\ref{e:mcadam} and \ref{e:mori} below show.
First, we characterize, in the strong Mori case, when $D$ weakly
Matlis implies $D[X]$ weakly Matlis. Recall that a domain $R$ is a
\emph{strong Mori domain} if it satisfies the ascending chain
condition on $w$-ideals.  We need the following facts from
\cite{wm}.  A domain $R$ is a strong Mori domain if and only if
$R$ has $t$-finite character and $R_M$ is Noetherian for each
maximal $t$-ideal $M$ of $R$ \cite[Theorem 1.9]{wm}; a strong Mori
domain satisfies the principal ideal theorem \cite[Corollary
1.11]{wm}; and if $R$ is a strong Mori domain, then so is $R[X]$
\cite[Theorem 1.13]{wm}.  Our next result also uses the fact that
in any Mori domain, each nonzero element is contained in at most
finitely many prime $t$-ideals \cite[Theorem 2.1]{hlv}.

\begin{proposition} \label{p:wmnoe} Suppose that $D$ is a
strong Mori (e.g., Noetherian) domain. Then $D[X]$ is weakly
Matlis if and only if $D$ is a weakly Matlis domain with at most
one maximal $t$-ideal of height greater than 1.
\end{proposition}
\begin{proof} Suppose that $D[X]$ is weakly Matlis.  Then
$D$ is weakly Matlis by Proposition~\ref{wm}.  Suppose that $M,N$
are maximal $t$-ideals of $D$ of height greater than 1. Pick any
nonzero element $a \in M \cap N$, and then choose $b \in M \cap N$
such that $b$ is not in any of the height one primes containing
$a$. Then $aX-b \notin P[X]$ for every height one prime $P$ of
$D$. Hence by the principal ideal theorem, the upper to zero
$Q=(aX-b)K[X] \cap D[X]$ is the only minimal prime of $aX-b$.
However, this implies that $Q \sub M[X] \cap N[X]$, contradicting
that $D[X]$ is weakly Matlis.

The converse follows easily from Proposition~\ref{wm} (and the
description of $\tmax(D[X])$ given in \cite[Proposition 1.1]{hz}).
\end{proof}

\bigskip


We now give an example of a weakly Matlis (in fact, $h$-local)
Noetherian domain with two maximal $t$-ideals of height two.

\begin{example} \label{e:mcadam}
In \cite{m}, McAdam gives an example of a Noetherian domain $T$
having exactly 2 maximal ideals $M,N$, both of height 2, such that
$M \cap N$ does not contain a nonzero prime.  Thus $T$ is
$h$-local.   We would like to transform the example, if necessary,
to ensure that the maximal ideals are $t$-ideals.  We do this by
means of a pullback.  Note that in McAdam's example, it is
possible to arrange that $T/M$ and $T/N$ are the same field $k$,
with $k$ arbitrary. Thus we may assume that $k$ contains a
subfield $F$ with $[k:F] < \infty$. Now let $D$ be the domain
arising from the following pullback diagram of canonical
homomorphisms:

$$  \begin{CD}
        D   @>>>    F \times F\\
        @VVV        @VVV    \\
        T  @>>>   T/(M \cap N) \cong k \times k\\
\end{CD}$$
\bigskip

\nd (The downward maps are containments.) It is well known that
$D$ is a Noetherian domain with spectrum homeomorphic to that of
$T$.  In particular, $D$ is $h$-local and therefore weakly Matlis.
We also have that the maximal ideals $M_0=M \cap D$ and $N_0 = N
\cap D$  of $D$ are (divisorial and therefore) $t$-ideals.  (For
example, if $m \in M \setminus D$, then $mN_0 \sub mN \sub D$, so
that $m \in N_0^{-1} \setminus D$.  Thus $N_0$ is divisorial.)  By
Proposition~\ref{p:wmnoe}, $D[X]$ is not weakly Matlis.
\end{example}

Next, we show that the ``strong Mori'' hypothesis in
Proposition~\ref{p:wmnoe} cannot be weakened to ``Mori''. The
following is an example of a one-dimensional semilocal (hence
automatically $h$-local and therefore weakly Matlis) Mori domain
$D$ for which $D[X]$ is not weakly Matlis.

\begin{example} \label{e:mori}
 Let $k$ be a field, and let $u,t$ be
indeterminates. Set $V_1=k(u)[t]_{(t)}$ and $V_2=k(u)[t]_{(t+1)}$.
Note that $V_1=k(u)+tV_1$ and $V_2=k(u)+(t+1)V_2$.  Set
$D_1=k+tV_1$ and $D_2=k+(t+1)V_2$.  Then for $i=1,2$, $D_i$ is a
local, one-dimensional  integrally closed Mori domain
\cite[Proposition 3.4]{bar}. Note that $K=k(u,t)$ is the common
quotient field of all these domains. Moreover, if $P_i=(X-u)K[X]
\cap D_i[X]$, then $P_i \subseteq N_i[X]$, where $N_i$ is the
maximal ideal of $D_i$. (This is well known. If $P \nsubseteq
N_i[X]$, then $u$ satisfies a polynomial over $D_i$ with a unit
coefficient; however, since $D_i$ is local and integrally closed,
the $u,u^{-1}$-lemma would then imply that $u$ or $u^{-1}$ is in
$D_i$, a contradiction.) Now set $D=D_1 \cap D_2$.  It is clear
that $t \in D_1$, and, since $t=-1+(t+1)$, we also have $t \in
D_2$.  Hence $t, t+1 \in D$.  It is now easy to see that $t(t+1)u
\in N_1 \cap N_2 \subseteq D$, so that $D$ has quotient field $K$.
Set $M_i=N_i \cap D$.  We claim that $M_1,M_2$ are the only
maximal ideals of $D$.  This follows from the fact that any
element of $D \setminus (M_1 \cup M_2)$ is a unit in both $D_1$
and $D_2$, hence also a unit in $D$. Moreover, we have
$D_{M_i}=D_i$ for $i=1,2$ by \cite[Corollary 8]{pr}. It follows
that $D$ is a one-dimensional integrally closed Mori domain (as
the intersection of two Mori domains) with exactly two maximal
ideals. Therefore, $D$ is automatically weakly Matlis. However, it
is easy to see that $P=(X-u)K[X] \cap D[X] \subseteq M_1[X] \cap
M_2[X]$; therefore, since $M_1[X]$ and $M_2[X]$ are maximal
$t$-ideals of $D[X]$, $D[X]$ is not weakly Matlis. (In fact,
$(X-u^j)K[X] \cap D[X] \subseteq M_1[X] \cap M_2[X]$ for each
nonzero integer $j$, so $M_1[X] \cap M_2[X]$ actually contains
infinitely many prime ideals.) Of course, $D$ is not strong Mori.
\end{example}


\section{$w$-divisoriality in polynomial rings}

In this section, we address

\begin{question} \label{q:wdiv} If $D$ is a $w$-divisorial domain,
is $D[X]$ also $w$-divisorial?
\end{question}

\medskip
It is convenient to recall that a domain $D$ is divisorial if and
only if it is $h$-local and $D_{M}$ is divisorial for each maximal
ideal $M$ (Theorem~\ref{t:bs}).  Similarly, $D$ is $w$-divisorial
if and only if it is weakly Matlis and $D_{M}$ is divisorial for
each maximal $t$-ideal $M$ \cite[Theorem 1.5]{GE}.

\begin{proposition} \label{p:divis}  The following conditions are equivalent
for a domain $D$.
    \begin{itemize}
    \item[(1)] $D[X]$ is $w$-divisorial.
 \item[(2)] $D[X]_{N_v}$ is divisorial.
 \end{itemize}

In case these equivalent conditions hold, $D$ is $w$-divisorial.

\end{proposition}

\begin{proof}  By Proposition~\ref{wm}, $D[X]$ is
    weakly Matlis if and only if $D[X]_{N_v}$ is $h$-local.

If $P$ is an upper to zero in $D[X]$, then $D[X]_{P}$ is a DVR.
Also, for each  $Q \in \op{Max} (D[X]_{N_v})$, we have $Q=
MD[X]_{N_v}$ with $M\in \tmax(D)$ and $(D[X]_{N_v})_{Q}=
D[X]_{MD[X]}$. It follows that $D[X]$ is $t$-locally divisorial if
and only if $D[X]_{N_v}$ is locally divisorial.  Thus (1) $\ra$
(2) by Theorem~\ref{t:bs}, and (2) $\ra$ (1) by \cite[Theorem
1.5]{GE}.

 Now assume that $D[X]_{N_v}$ is divisorial.
For each nonzero ideal $J$ of $D$, we have
    $J_{w}=JD[X]_{N_v} \cap D$ \cite[Proposition 3.4]{FL} and $J_{v}D[X]_{N_v} =
    (JD[X]_{N_v})_{v}$ by \cite[Corollary 2.3]{K}.
    Thus, if  $I$ is a nonzero ideal of $D$, we have
    $$I_{w}´=ID[X]_{N_v} \cap D = (ID[X]_{N_v})_{v}´\cap D=I_{v}D[X]_{N_v} \cap
    D=I_{v}.$$ Thus $D$ is $w$-divisorial.
 \end{proof}

To simplify notation, for an ideal $I$ and a maximal $t$-ideal $M$
of $D$, we set $I_M(X)= ID_{M}(X)$. We have
$I_M(X)=I_M[X]_{M_M[X]}= I[X]_{M[X]}$.

\begin{lemma} \label{loc} If $D$ is a  weakly Matlis domain and  $M \in
\tmax(D)$, then $(I:J)_M(X)=(I_M(X):J_M(X))$, for each pair of
$t$-ideals $I$, $J$ of $D$. In particular, if $J$ is divisorial,
then $J_M(X)$ is divisorial.
\end{lemma}
\begin{proof}
If  $D$ is a weakly Matlis domain, then for any pair of $t$-ideals
$I$, $J$ and each maximal $t$-ideal $M$ of $D$, we have
$(I:J)_M=(I_M:J_M)$ \cite[Corollary 5.2]{AZ}; in particular,
$(JD_M)_v=J_vD_M$.

Now for $M \in \tmax(D)$, we have $MD_M \in \tmax(D_M)$ by
\cite[Lemma 3.3]{amz}.  Hence $D_M$ is weakly Matlis, and
therefore so is $D_M[X]$ (Corollary~\ref{c:umt}). In addition,
$M_M[X]$ is a maximal $t$-ideal of $D_M[X]$. Thus
$(I:J)_M(X)=(I:J)_M[X]_{M_M[X]}= (I_M:J_M)[X]_{M_M[X]}=
(I_M[X]:J_M[X])_{M_M[X]}=(I_M[X]_{M_M[X]}:J_M[X]_{M_M[X]})=(I_M(X):J_M(X))$.
\end{proof}

An ideal $J$ of $D$ which is $m$-canonical in $(J:J)$ is called
\emph{quasi-$m$-canonical} in \cite [Section 2.3.4]{P}; by
\cite[Lemma 2.56]{P}, $J$ is quasi-$m$-canonical if and only if
$(J:(J:I))=I(J:J)$ for each nonzero ideal $I$ of $D$.

As already noted, $D[X]$ is $w$-divisorial if and only if $D[X]$
is weakly Matlis and $D[X]_Q$ is divisorial for each maximal
$t$-ideal $Q$ of $D[X]$.  The following results shortens the list
of maximal $t$-ideals which have to be checked.

\begin{proposition} \label{wdiv} Assume that $D$ is a
$w$-divisorial  domain. The following statements are equivalent.
    \begin{itemize}
        \item[(1)]  $D[X]$ is $w$-divisorial.
\item[(2)]  $D[X]$ is weakly Matlis and $D_{M}(X)$ is divisorial
for each $M \in \tmax(D)$ such that $(M:M)=(D:M)$.
\item[(3)] $D[X]$ is weakly Matlis and
$M_{M}(X)$ is quasi-$m$-canonical for each $M \in \tmax(D)$ such
that $(M:M)=(D:M)$.
\end{itemize}
\end{proposition}

\begin{proof}
$(1)\ra(2)$.  This follows from \cite[Theorem 1.5]{GE} and the
fact that $M[X] \in \tmax(D[X])$ for each $M \in \tmax(D)$.

$(2)\ra(1)$. Let $Q$ be a maximal $t$-ideal of $D[X]$. If $Q$ is
an upper to zero, then $D[X]_Q$ is a DVR and hence divisorial.
Hence we may as well assume that $Q=M[X]$ for some maximal
$t$-ideal $M$ of $D$ \cite[Proposition 1.1]{hz}. By \cite[Theorem
1.5]{GE} $D_{M}$ is divisorial. Suppose that $M$ is $t$-invertible
in $D$. Then $D_M$ is a valuation domain \cite[Lemma 3.1]{GE}, and
hence each ideal of $D_M(X)$ is extended from $D$ \cite[Theorem
3.1]{K}. If $I = J(X)$ is an extended ideal of $D_M(X)$, with
$J\sub D_{M}$, then $J(X)=J_{v}(X)= J(X)_{v}$ is divisorial
\cite[Corollary 2.3]{K}. Thus $D[X]_Q$ is divisorial in this case.
If $M$ is not $t$-invertible, then $(M:M)=(D:M)$, and this case is
covered by the hypothesis.

$(2)\ra(3)$. Let $M \in \tmax(D)$ be such that $(M:M)=(D:M)$. By
Lemma \ref{loc}, we have that $(M_M(X):M_M(X))=(D_M(X):M_M(X))$.
In addition, $M_M(X)$ is a divisorial ideal; thus  $S =
(M_M(X):M_M(X))$ is a proper overring of $D_M(X)$. Since $D_M(X)$
is divisorial and $(D_M(X):S) = M_M(X)\neq 0$, then $M_M(X)$ is an
$m$-canonical ideal of $S$ by \cite[Proposition 5.1]{hhp}.

$(3)\ra (2)$  Let $M \in \tmax(D)$ with $(M:M)=(D:M)$. Then $D_M$
is divisorial, whence $T:=(M_M:M_M)$ is a 2-generated $D_M$-module
by Theorem~\ref{t:divequiv}.  Since $D$ is weakly Matlis,
$(M:M)_M=(M_M:M_M)$ by \cite[Corollary 5.2]{AZ}. Hence by
Lemma~\ref{loc}, $T(X)=(M_M:M_M)(X)=(M:M)_M(X)=(M_M(X):M_M(X))$.
By hypothesis, $M_M(X)$ is $m$-canonical in $T(X)$.  Hence by
Theorem~\ref{t:divequiv}, we need only show that $T(X)$ is a
2-generated $D_M(X)$-module.  However, since $T(X)$ is integral
over $D_M(X)$, \cite[Remark 1]{gh} implies that $T(X)=T[X]_{D_M[X]
\setminus M_M[X]}$, and the desired conclusion follows easily.

\end{proof}

We can now show that $w$-divisoriality transfers from $D$ to
$D[X]$ in two cases (of course, $w$-divisoriality of $D[X]$
implies that of $D$ by Proposition~\ref{p:divis}).

    \begin{corollary}\label{PVMD}
    $D$ is an integrally closed $w$-divisorial domain if and only
if $D[X]$ is an integrally closed $w$-divisorial domain.
\end{corollary}
\begin{proof} An integrally closed $w$-divisorial domain
is a weakly Matlis P$v$MD (hence a UMT-domain) with each maximal
$t$-ideal $t$-invertible \cite[Theorem 3.3]{GE}. Hence we conclude
by applying Corollary~\ref{c:umt} and Proposition \ref{wdiv}.
\end{proof}

    \begin{proposition}  If $D$ is a Mori domain, then
    $D$ is $w$-divisorial if and only if $D[X]$ is
$w$-divisorial.
\end{proposition}

\begin{proof} Suppose that $D$ is a $w$-divisorial Mori domain. Then
$D$ is a strong Mori domain of $t$-dimension 1 by \cite[Corollary
4.3]{GE}. Proposition~\ref{p:wmnoe} then yields that $D[X]$ is
weakly Matlis.  Let $M$ be a maximal $t$-ideal of $D$.  Then $D_M$
is Noetherian by \cite[Theorem 1.9]{wm}.  Hence $T=(D_M:MD_M)$ is
a 2-generated $D_M$-module by \cite[Theorem 3.8]{ma}.  It follows
as in the proof of Proposition~\ref{wdiv} that $T(X)$ is a
2-generated $D_M(X)$-module.  Another application of \cite[Theorem
3.8]{ma} then yields that $D_M(X)$ is divisorial. Hence $D[X]$ is
$w$-divisorial by \cite[Theorem 1.5]{GE} or
Proposition~\ref{wdiv}.

\end{proof}

We close this section with further discussion of
Question~\ref{q:wdiv}. It is helpful to recall an old question and
to add another one:

\begin{question} \label{q:heinzer} Is the integral closure of a
divisorial domain a Pr\"ufer domain?
\end{question}

\begin{question} \label{q:umt} Is a $w$-divisorial domain a
UMT-domain?
\end{question}

Question~\ref{q:heinzer} was asked by Heinzer in \cite{h}.  It is
fairly easy to show that Questions~\ref{q:heinzer} and \ref{q:umt}
are equivalent.  First, suppose that Question~\ref{q:heinzer} has
a positive answer, and let $D$ be a $w$-divisorial domain.  By
\cite[Theorem 1.5]{GE} $D_M$ is divisorial for each maximal
$t$-ideal $M$ of $D$, whence by assumption each such $D_M$ has
Pr\"ufer integral closure.  We then have that $D$ is a UMT-domain
by \cite[Theorem 2.4]{dhlrz}.  Conversely, suppose that
Question~\ref{q:umt} has a positive answer, and let $D$ be a
divisorial domain.  Then $D$ is automatically $w$-divisorial and
therefore UMT by assumption. If $M$ is a maximal ideal of $D$,
then $M$ is divisorial and hence a maximal $t$-ideal.  Again by
\cite[Theorem 2.4]{dhlrz} $D$ has Pr\"ufer integral closure.

Now consider Question~\ref{q:wdiv} again.  If $D$ is
$w$-divisorial, and we assume that this implies that $D$ is UMT
(that is, that the equivalent Questions~\ref{q:heinzer} and
\ref{q:umt} have positive answers), then ($D$ is weakly Matlis and
therefore) $D[X]$ is weakly Matlis by Corollary~\ref{c:umt}.  To
show that $D[X]$ is $w$-divisorial, it would therefore suffice to
show that $D[X]_{M[X]}=D_M[X]_{MD_M[X]}$ is divisorial for each
maximal $t$-ideal $M$ of $D$.  Since $D_M$ is divisorial, this
leads to the following

\begin{question} \label{q:div} If $(D,M)$ is a local divisorial
domain, is $D[X]_{M[X]} (=D(X))$ divisorial?
\end{question}

Let $(D,M)$ be a local divisorial domain.  According to
Proposition~\ref{wdiv} $(1) \lra (3)$, we may assume that $D
\subsetneqq T:= (M:M)=(D:M)$, and we have to show that $M(X)$ is
$m$-canonical in $T(X)$.  We are able to do this in two of the
three cases mentioned in Theorem~\ref{t:bs}:

\bigskip

\noindent Case 3.  $T$ is a valuation domain with maximal ideal
$M$.  Then $T(X)$ is also a valuation domain, and its maximal
ideal $M(X)$ is $m$-canonical by \cite[Proposition 6.2]{hhp}.
\bigskip

\noindent Case 2. $T$ is a Pr\"ufer domain with exactly two
maximal ideals $N_1$ and $N_2$ with $M = N_1 \cap N_2$.  Note that
$T$ is necessarily the integral closure of $D$. By \cite[Theorem
6.7]{hhp} (or by \cite[Proposition 2.9]{b}), $T$ is $h$-local.  It
is then easy to see that $T(X)$ is $h$-local. We claim that $M(X)
= N_1(X) \cap N_2(X)$ is an $m$-canonical ideal of $T(X)$.  If
both $N_1$ and $N_2$ are noninvertible, this follows from the
proof of \cite[Theorem 6.7]{hhp}.  If only one, say ($N_1$ hence)
$N_1(X)$ is noninvertible, then the proof shows that $N_2(X)$ is
$m$-canonical.  But multiplication by a principal ideal is
harmless, so that in this case $M(X)=N_1(X)N_2(X)=N_1(X) \cap
N_2(X)$ is also $m$-canonical. Finally, if both $N_1(X)$ and
$N_2(X)$ are principal, then the proof guarantees that $T(X)$
itself is (divisorial hence) $m$-canonical, and again it follows
easily that $M(X)$ is $m$-canonical. \bigskip

\end{document}